\documentclass[11pt]{article}

\usepackage[margin=1in]{geometry}
\usepackage{amsmath,amssymb,amsthm,mathtools}
\usepackage[T1]{fontenc}
\usepackage[utf8]{inputenc}
\usepackage{lmodern}
\usepackage{hyperref}
\usepackage{xcolor}
\usepackage{bbm}
\usepackage{bm}
\usepackage{upgreek}

\hypersetup{hidelinks}
\theoremstyle{plain}
\newtheorem{theorem}{Theorem}[section]
\newtheorem{lemma}[theorem]{Lemma}
\newtheorem{corollary}[theorem]{Corollary}

\theoremstyle{definition}
\newtheorem{definition}[theorem]{Definition}
\newtheorem{remark}[theorem]{Remark}

\numberwithin{equation}{section}

\newcommand{\R}{\mathbb{R}}

\newcommand{\E}{\mathbb{E}}
\newcommand{\MVME}[1]{\if\relax\detokenize{#1}\relax\mathrm{MVME}\else{}_{#1}\mkern1.5mu\mathrm{MVME}\fi}
\newcommand{\MME}[1]{\if\relax\detokenize{#1}\relax\mathrm{MME}^{*}\else{}_{#1}\mkern1.5mu\mathrm{MME}^{*}\fi}
\newcommand{\MVPH}[1]{\if\relax\detokenize{#1}\relax\mathrm{MVPH}\else{}_{#1}\mkern1.5mu\mathrm{MVPH}\fi}
\newcommand{\MPH}[1]{\if\relax\detokenize{#1}\relax\mathrm{MPH}^{*}\else{}_{#1}\mkern1.5mu\mathrm{MPH}^{*}\fi}
\newcommand{\ME}{\mathrm{ME}}
\newcommand{\PH}{\mathrm{PH}}
\newcommand{\Tr}{\mathrm{Tr}}
\newcommand{\diag}{\mathrm{diag}}
\newcommand{\balpha}{\boldsymbol{\upalpha}}
\newcommand{\boldeta}{\boldsymbol{\upeta}}

\begin{document}

\title{Equivalence and Separation for Multivariate Matrix-Exponential and Phase-Type Distribution Classes}
\author{Oscar Peralta}
\date{}
\maketitle

\begin{abstract}
We resolve two questions left open by Bladt and Nielsen~\cite{bladtnielsen2010} concerning multivariate families of matrix-exponential and phase-type distributions. First, in the matrix-exponential case, the projection-defined class \(\MVME{}\) coincides with Kulkarni's algebraic class \(\MME{}\). Our proof combines a multivariate state-space realization theorem with elementary augmentations that put the realization into Kulkarni's form. Thus every proper rational multivariate Laplace transform has a finite-dimensional Kulkarni-type representation once Markovian sign constraints are removed. Second, in the phase-type setting, the inclusion \(\MPH{}\subsetneq\MVPH{}\) is strict from the trivariate case onward. The separation is obtained through a factorization condition for \(\MPH{}\) that appears not to have been previously identified in the PH literature. A Wishart trace distribution belongs to \(\MVPH{}\) but fails this condition, hence providing the required example outside \(\MPH{}\). The example also shows that projection-based multivariate phase-type laws may have density and support geometry that are absent from the usual univariate theory.
\end{abstract}

\section{Introduction}

Phase-type (PH) and matrix-exponential (ME) distributions on the positive half-line form two useful and well-established families in applied probability. In the phase-type case a non-negative random variable is defined as the absorption time of a finite-state continuous-time Markov chain; in the matrix-exponential case the Markovian constraint is relaxed to requiring a rational Laplace--Stieltjes transform. PH distributions have long been a standard matrix-analytic tool, particularly in queueing, risk, and reliability models~\cite{neuts1981,asmussen2003,bladtnielsen2017}. ME distributions may be regarded as algebraic generalizations of PH distributions, retaining rational-transform structure while dropping the underlying Markovian sign constraints~\cite{asmussen1997renewal,bladtneuts2003,peralta2023}. Both families are dense, in the sense of weak convergence, in the class of all distributions on $\R_{+}$. Their analytical tractability, including closure under convolutions and finite mixtures, together with explicit matrix formulas for renewal quantities and related transform calculations, is a main reason for their broad use in matrix-analytic modelling.

In several variables, however, the passage from PH and ME distributions on \(\R_{+}\) to distributions on \(\R_{+}^{n}\) is not canonical. Recent work has therefore developed tractable multivariate subclasses designed for particular modelling purposes, including affine-mixture constructions for matrix-exponential risks~\cite{cheung2022affine}, multivariate phase-type models for loss modelling~\cite{bladt2023loss}, dependent phase-type models in risk theory~\cite{bladt2019parisian,peralta2023ruin}, and common-shock constructions based on matrix distributions~\cite{bladtperaltayslas2026}. Those models provide explicit and estimable dependence structures. The present paper addresses a more structural question: whether the most natural projection-based definitions have genuine finite-state reward interpretations. Rather than proposing a new subclass, we compare two broad ways of extending the PH and ME frameworks to several dimensions.

The first is the reward-based construction introduced by Kulkarni~\cite{kulkarni1989}, as a generalization of~\cite{assaf1984}. It defines a random vector \(\bm{X}=(X_{1},X_{2},\ldots,X_{n})^{\top}\in\R_{+}^{n}\) from an underlying finite-state continuous-time Markov chain, where each component \(X_j\) accumulates non-negative reward at a state-dependent rate until absorption. With the full Markovian constraints on the generator and non-negativity of the reward matrix, this gives the class \(\MPH{}\) in \(n\) variates; dropping all sign constraints on both matrices gives the algebraic generalization \(\MME{}\) in \(n\) variates. The second is the projection-based construction formalized by Bladt and Nielsen~\cite{bladtnielsen2010}, where membership is determined by all positive one-dimensional projections. A vector \(\bm{X}\) belongs to \(\MVPH{}\) in \(n\) variates if \(\langle \mathbf{a},\bm{X}\rangle\) is univariate phase-type for every \(\mathbf{a}\in\R_{+}^{n}\setminus\{0\}\), and it belongs to \(\MVME{}\) in \(n\) variates if every such projection is matrix-exponential, equivalently, if the joint Laplace transform is a proper rational function of the transform variables \((s_1,\ldots,s_n)\).

The inclusions $\MPH{}\subseteq\MVPH{}$ and $\MME{}\subseteq\MVME{}$, for a fixed variate dimension, are immediate from the definitions (Remark~\ref{rem-trivial}). The central question is whether these inclusions express genuine equivalences or whether the projection-defined classes are strictly larger. The difference is substantive: a finite-state reward representation provides a pathwise stochastic construction and the associated matrix-analytic calculus, whereas a projection-defined law may be accessible only through its one-dimensional marginals. It was shown in~\cite{bladtnielsen2010} that some elements of $\MVME{}$ require a state size strictly larger than the minimal algebraic degree of their Laplace transform in any $\MME{}$ realization, but this does not by itself give a strict separation between the two classes.

We resolve both questions. For the matrix-exponential classes, we prove the equivalence between \(\MVME{}\) and \(\MME{}\) in \(n\) variates for all $n\ge 2$; the diagonal structure of Kulkarni's resolvent imposes no genuine restriction once positivity is dropped. The proof uses a realization theorem from linear systems theory~\cite{alpay2003,cheng2011realization}, which provides a multivariate analogue of the one-variable ME representation by expressing a rational transform through a finite-dimensional resolvent in which several variables enter linearly. The passage from this systems-theoretic form to Kulkarni's diagonal resolvent is not direct and requires additional state-space transformations. Thus every proper rational multivariate Laplace transform admits a Kulkarni-form realization. In particular, the projection-defined ME class is not a looser transform class: it is exactly the finite-dimensional algebraic reward class.

For the phase-type classes, we prove the separation \(\MPH{}\subsetneq\MVPH{}\) in \(n\) variates for all $n\ge 3$. The main intermediate result is a factorization criterion for Kulkarni-type Markovian representations. The Markovian structure imposes a concrete algebraic signature on the denominator of the joint transform: its highest-degree part cannot be arbitrary, but must split into elementary real linear factors with non-negative coefficients. This condition is easy to miss from the probabilistic definition, but it gives an effective criterion for excluding distributions from \(\MPH{}\).

A three-dimensional Wishart trace distribution satisfies the projection requirement defining \(\MVPH{}\), but its denominator fails this factorization condition. It therefore lies between the two classes, belonging to \(\MVPH{}\) but not to \(\MPH{}\). Appending independent exponentials extends the separation to all $n\ge 3$. Wishart matrices originate in multivariate statistics as matrix-valued analogues of gamma or chi-square variables~\cite{wishart1928}; here they are used only through their determinant-form Laplace transform. To the best of our knowledge, this link between Wishart trace functionals and the projection-based multivariate phase-type class has not previously been developed in the literature. The same example shows that a rational multivariate Laplace transform may correspond to a density with a genuine square-root singularity and a non-polyhedral support, features absent from the univariate theory.

These results clarify what is, and is not, lost when multivariate PH and ME laws are defined through their one-dimensional projections. In the ME case, no structure is lost: every projection-defined law admits a finite-dimensional algebraic reward representation. In the PH case, by contrast, the projection property alone does not guarantee the existence of an underlying finite-state Markov reward process. The separation thus marks a modelling distinction between having PH projections and possessing a finite-state pathwise construction. In matrix-analytic modelling, reward representations carry structure beyond the transform itself: simulation, conditioning, and computational tractability.

The paper is organized as follows. Section~\ref{sec-defs} collects the definitions and preliminary inclusions. Section~\ref{sec-mvme} proves the matrix-exponential equivalence between \(\MVME{}\) and \(\MME{}\) in \(n\) variates. Section~\ref{sec-separation} proves the phase-type separation, establishing the factorization criterion for \(\MPH{}\) representations and constructing an explicit counterexample from a Wishart trace distribution, with an extension to all $n\ge 3$. Section~\ref{sec-remarks} collects concluding remarks.

\section{Definitions and preliminaries}
\label{sec-defs}

We fix $n\ge 2$ throughout and work on $\R_{+}^{n}=[0,\infty)^{n}$. Vectors in $\R^{n}$ are regarded as column vectors unless explicitly stated otherwise. Deterministic vectors and matrices are written in upright boldface, as in $\mathbf{s}$, $\mathbf{x}$, $\mathbf{t}$, $\mathbf{T}$, and $\mathbf{K}$. Deterministic Greek vectors are written in upright bold Greek, for instance $\balpha$ and $\boldeta$. Random vectors and random matrices are written with capital bold italic letters, such as $\bm{X}$, $\bm{Y}$, and $\bm{Z}$. For $\mathbf{s}=(s_{1},\ldots,s_{n})^{\top}\in\R^{n}$ and \(\mathbf{x}=(x_{1},\ldots,x_{n})^{\top}\in\R^{n}\), write $\langle \mathbf{s},\mathbf{x}\rangle=\sum_{j=1}^{n}s_{j}x_{j}$.

For a deterministic vector $\mathbf{v}\in\R^{m}$, let $\Delta(\mathbf{v})=\diag(\mathbf{v})$ denote the $m\times m$ diagonal matrix with entries $v_{1},\ldots,v_m$. The symbol $\mathbf{1}_r$ denotes the all-ones column vector of size $r$ and we suppress the subscript when the dimension is clear. $\mathbf{I}_m$ denotes the $m\times m$ identity matrix, $\mathbf{0}_m$ the zero column vector in $\R^m$, and $\mathbf{0}_{r\times q}$ the $r\times q$ zero matrix. For $\mathbf{s}\in\R^{n}$ and $\mathbf{K}\in\R^{m\times n}$, let $(\mathbf{K}\mathbf{s})_{i}=\sum_{j=1}^{n}\mathbf{K}_{ij}s_{j}$ and write $\Delta(\mathbf{K}\mathbf{s})=\diag(\mathbf{K}\mathbf{s})$.

\begin{definition}[ME and PH]\label{def-univariate}
A non-negative random variable $Z$ belongs to the \emph{matrix-exponential class} $\ME$ if its Laplace--Stieltjes transform $\mathcal{L}_{Z}(s)=\E[e^{-sZ}]$ is a rational function of $s$. Concretely, there exist $m\ge 1$, a row vector $\balpha\in\R^{1\times m}$, a Hurwitz-stable matrix $\mathbf{T}\in\R^{m\times m}$, and a column vector $\mathbf{t}\in\R^{m\times 1}$ such that, with
\[
  p_{0}=1-\balpha(-\mathbf{T})^{-1}\mathbf{t}\ge 0,
\]
the transform is
\begin{equation}\label{eq-univariate-me}
  \mathcal{L}_{Z}(s) =p_{0}+\balpha(s\mathbf{I}_{m}-\mathbf{T})^{-1}\mathbf{t}.
\end{equation}
Equivalently, apart from a possible atom $p_0$ at zero, the absolutely continuous component has density
\[
  f_Z(x)=\balpha e^{x\mathbf{T}}\mathbf{t},\qquad x>0,
\]
whenever the right-hand side is non-negative and integrates to $1-p_0$.
Here Hurwitz-stable means that every eigenvalue of $\mathbf{T}$ has strictly negative real part; this ensures that the resolvent in~\eqref{eq-univariate-me} is well behaved on the right half-plane.

$Z$ is \emph{phase-type} ($\PH$) if additionally $\mathbf{T}$ is the subgenerator of a transient continuous-time Markov chain. $\mathbf{T}_{ii}<0$, $\mathbf{T}_{ij}\ge 0$ for $i\ne j$, the row sums satisfy $\mathbf{T}\mathbf{1}\le 0$ component-wise, $\balpha\ge 0$ with $\balpha\mathbf{1}\le 1$, and $\mathbf{t}=-\mathbf{T}\mathbf{1}\ge 0$. The defect $1-\balpha\mathbf{1}$ is the probability of immediate absorption, equivalently the atom of $Z$ at zero.
\end{definition}

We now describe the two multivariate approaches used throughout the paper. The first is the \emph{reward-based construction} of Kulkarni~\cite{kulkarni1989}, in which a finite transient continuous-time Markov chain accumulates state-dependent rewards until absorption. The second is the \emph{projection-based construction} of Bladt and Nielsen~\cite{bladtnielsen2010}, in which membership is tested through all positive linear projections.

\paragraph{Notation for variate dimension.}
Throughout the paper, the left subscript on a multivariate distribution class records the variate dimension, that is, the number of coordinates of the random vector. Thus
\[
  \MPH{n},\qquad \MME{n},\qquad \MVPH{n},\qquad \MVME{n}
\]
denote the corresponding $n$-variate classes. This convention is separate from the order of a particular matrix representation, which is denoted by $m$ when it is needed and is otherwise left implicit.

\begin{definition}[$\MPH{n}$ and $\MME{n}$]\label{def-kulkarni}
An \(m\)-state Markovian reward-based representation is specified by a triplet \((\balpha,\mathbf{T},\mathbf{K})\), where \(\balpha\) is the initial row vector, \(\mathbf{T}\) is the transient subgenerator, and \(\mathbf{K}\) is the reward matrix. In this Markovian case, \(\mathbf{t}\) and \(p_0\) are determined by \(\mathbf{t}=-\mathbf{T}\mathbf{1}\) and \(p_0=1-\balpha\mathbf{1}\). A random vector $\bm{X}\in\R_{+}^{n}$ belongs to $\MPH{n}$ if there exist $m\ge 1$, a sub-probability row vector $\balpha\in\R_{+}^{1\times m}$, a transient subgenerator $\mathbf{T}\in\R^{m\times m}$, and a non-negative reward matrix $\mathbf{K}\in\R_{+}^{m\times n}$ such that
\begin{equation}\label{eq-kulkarni}
  L_{\bm{X}}(\mathbf{s})=\E\bigl[e^{-\langle \mathbf{s},\bm{X}\rangle}\bigr] =p_{0}+\balpha\bigl(-\mathbf{T}+\Delta(\mathbf{K}\mathbf{s})\bigr)^{-1}\mathbf{t},
\end{equation}
where $\mathbf{t}=-\mathbf{T}\mathbf{1}$ and $p_{0}=1-\balpha\mathbf{1}\ge 0$. In probabilistic terms, the underlying continuous-time Markov chain starts according to $\balpha$ unless it is absorbed at time zero with probability $p_0$. If \(J_t\) denotes the transient state occupied at time \(t\), then the \(j\)-th coordinate is the accumulated reward
\[ 
  X_j=\int_0^\tau \mathbf{K}_{J_t j}\,\mathrm{d}t,
\]
where \(\tau\) is the absorption time. Thus, while the chain is in transient state \(i\), coordinate \(j\) accumulates reward continuously at rate \(\mathbf{K}_{ij}\ge0\). The vector $\bm{X}$ belongs to $\MME{n}$ if its transform admits a representation of the form~\eqref{eq-kulkarni} with \(\mathbf{T}\) Hurwitz-stable, but without the Markovian and sign constraints on \(\balpha\), \(\mathbf{T}\), \(\mathbf{K}\), or \(\mathbf{t}\), and subject only to \(L_{\bm X}\) being the Laplace transform of a genuine probability measure on \(\R_+^n\).
\end{definition}

\begin{remark}\label{rem-normalization}
The definition of $\MME{n}$ imposes no sign constraint on $\balpha$, $\mathbf{T}$, or $\mathbf{K}$, and in particular does not require the normalization $\mathbf{t}=-\mathbf{T}\mathbf{1}$, the condition $0\le\balpha\mathbf{1}\le 1$, or the non-negativity $\mathbf{K}\ge 0$. These are precisely the structural conditions that appear in the Markovian definition of $\MPH{n}$, but they place no genuine restriction on the algebraic class $\MME{n}$. The constructive proof of $\MVME{n}\subseteq\MME{n}$ given in Section~\ref{sec-mvme} does, however, produce representations that satisfy all three of them, yielding a subclass of $\MME{n}$ whose formal structure mirrors that of $\MPH{n}$. Such normalizations are useful beyond formal analogy: in more general reward systems with ME-like components, related constraints on the reward matrix, initial vector, and terminal column provide additional control over the stochastic interpretation and analytic tractability of the model; see, for example, \cite{bean2022rap}. The same remark applies in the univariate case: the $\ME$ class does not require $\mathbf{t}=-\mathbf{T}\mathbf{1}$ or $\balpha\mathbf{1}\in[0,1]$ as part of its definition, although equivalent normalized representations are available; see, for example, \cite{asmussen1997renewal}.
\end{remark}

The projection-based classes are defined by the behavior of the distribution under positive linear combinations.

\begin{definition}[$\MVME{n}$ and $\MVPH{n}$]\label{def-bn}
$\bm{X}\in\R_{+}^{n}$ belongs to $\MVME{n}$ if its joint Laplace transform $L_{\bm{X}}(\mathbf{s})$ is a proper rational function of $(s_{1},\ldots,s_{n})$. It is shown in~\cite{bladtnielsen2010} that this is equivalent to requiring $\langle \mathbf{a},\bm{X}\rangle\in\ME$ for every $\mathbf{a}\in\R_{+}^{n}\setminus\{0\}$. Further, we say that $\bm{X}$ belongs to $\MVPH{n}$ if $\langle \mathbf{a},\bm{X}\rangle\in\PH$ for every $\mathbf{a}\in\R_{+}^{n}\setminus\{0\}$.
\end{definition}

\begin{remark}\label{rem-trivial}
$\MPH{n}\subseteq\MVPH{n}$ and $\MME{n}\subseteq\MVME{n}$. For $\bm{X}\in\MPH{n}$ with parameters $(\balpha,\mathbf{T},\mathbf{K})$ and any $\mathbf{a}\in\R_{+}^{n}\setminus\{0\}$, the projected transform
\[
  L_{\bm{X}}(u\mathbf{a})=p_{0}+\balpha(u\Delta(\mathbf{K}\mathbf{a})-\mathbf{T})^{-1}\mathbf{t}
\]
is the transform of a univariate PH law. This is standard, including the treatment of zero reward rates; see~\cite{bladtnielsen2010} for details. The analogous inclusion \(\MME{n}\subseteq\MVME{n}\) follows because the same projection formula is rational in \(u\).
\end{remark}

The four classes, with the results proved below, are as follows:
\[
\begin{array}{c|c|c}
\text{class} & \text{defining feature} & \text{result in this paper} \\ \hline
\MME{n} & \text{algebraic reward representation} & =\MVME{n} \\
\MVME{n} & \text{all positive projections are ME} & =\MME{n} \\
\MPH{n} & \text{Markovian reward representation} &  \text{strictly smaller than }\MVPH{n}\text{ for }n\ge3 \\
\MVPH{n} & \text{all positive projections are PH} & \text{strictly larger than }\MPH{n}\text{ for }n\ge3.
\end{array}
\]

\section[Equivalence of MVME and MME]{Equivalence of $\MVME{n}$ and $\MME{n}$}
\label{sec-mvme}

The inclusion $\MME{n}\subseteq\MVME{n}$ is Remark~\ref{rem-trivial}. The non-trivial direction is $\MVME{n}\subseteq\MME{n}$. If $\bm{X}\in\MVME{n}$, then its Laplace transform can be written in reduced form as $L_{\bm{X}}(\mathbf{s})=P(\mathbf{s})/Q(\mathbf{s})$, where $P$ and $Q$ are polynomials with real coefficients, have no non-constant common factor, and satisfy $Q(0)\ne0$. The problem is to realize this arbitrary rational transform in Kulkarni's diagonal resolvent form.

The apparent restriction comes from the way the transform variables \(s_{1},\ldots,s_{n}\) enter a reward-based representation. They appear only through the diagonal term \(\Delta(\mathbf{K}\mathbf{s})\), so the denominator is produced by a determinant in which each state contributes one linear expression in \(\mathbf{s}\). This places Kulkarni's formula in a form that is not obviously comparable with a general multivariate rational denominator. The equivalence theorem shows that, once the Markovian sign constraints are removed, this diagonal structure does not restrict the resulting class. In other words, general rational Laplace transforms can be converted into Kulkarni form by passing through a finite-dimensional realization and then applying a series of state-space augmentations.

We first recall the version of the realization theorem from~\cite{alpay2003,cheng2011realization} needed below. We state it in the form used here, specialized to scalar rational functions with real coefficients in the same \(n\) variables as the variate dimension of the transform.

\begin{theorem}\label{thm-alpay}
Let $\widetilde{L}(\mathbf{s})$ be a rational function with real coefficients, analytic at the origin. Then there exist $N\ge 1$, a row vector $\boldeta\in\R^{1\times N}$, a column vector $\mathbf{b}\in\R^{N\times 1}$, and matrices $\mathbf{A}_{1},\ldots,\mathbf{A}_{n}\in\R^{N\times N}$ such that
\begin{equation}\label{eq-alpay}
  \widetilde{L}(\mathbf{s})=\boldeta\Bigl(\mathbf{I}_{N}-\sum_{j=1}^{n}s_{j}\mathbf{A}_{j}\Bigr)^{-1}\mathbf{b}.
\end{equation}
\end{theorem}

\begin{proof}
The multivariate realization theorem in \cite{cheng2011realization} gives a Fornasini--Marchesini realization of the form
\[
  \widetilde{L}(\mathbf{s})=d+\mathbf{c}\Bigl(\mathbf{I}_{\rho}-\sum_{j=1}^{n}s_{j}\mathbf{R}_{j}\Bigr)^{-1} \Bigl(\sum_{j=1}^{n}s_{j}\mathbf{b}_{j}\Bigr),
\]
where \(d\in\R\), \(\mathbf{c}\in\R^{1\times \rho}\), \(\mathbf{b}_{j}\in\R^{\rho\times 1}\), and \(\mathbf{R}_{j}\in\R^{\rho\times \rho}\). We convert this formula into the constant-input form~\eqref{eq-alpay} by adding one state. Set $N=\rho+1$ and define
\[
  \mathbf{A}_{j}=\begin{pmatrix}0&\mathbf{0}_{1\times \rho}\\ \mathbf{b}_{j}&\mathbf{R}_{j}\end{pmatrix}, \qquad \mathbf{b}=\begin{pmatrix}1\\\mathbf{0}_{\rho}\end{pmatrix}, \qquad \boldeta=\begin{pmatrix}d&\mathbf{c}\end{pmatrix}.
\]
Since
\[
  \mathbf{I}_{N}-\sum_{j=1}^{n}s_{j}\mathbf{A}_{j} =
  \begin{pmatrix}
    1 & \mathbf{0}_{1\times \rho}\\ -\sum_{j=1}^{n}s_{j}\mathbf{b}_{j} & \mathbf{I}_{\rho}-\sum_{j=1}^{n}s_{j}\mathbf{R}_{j}
  \end{pmatrix},
\]
the standard inverse formula for a block lower-triangular matrix \cite[Appendix~A.1]{bladtnielsen2017} gives
\[
  \Bigl(\mathbf{I}_{N}-\sum_{j=1}^{n}s_{j}\mathbf{A}_{j}\Bigr)^{-1}\mathbf{b} =
  \begin{pmatrix}
    1\\[2mm] \bigl(\mathbf{I}_{\rho}-\sum_{j=1}^{n}s_{j}\mathbf{R}_{j}\bigr)^{-1} \sum_{j=1}^{n}s_{j}\mathbf{b}_{j}
  \end{pmatrix}.
\]
Multiplying by $\boldeta$ yields
\[
  \boldeta\Bigl(\mathbf{I}_{N}-\sum_{j=1}^{n}s_{j}\mathbf{A}_{j}\Bigr)^{-1}\mathbf{b} = d+\mathbf{c}\Bigl(\mathbf{I}_{\rho}-\sum_{j=1}^{n}s_{j}\mathbf{R}_{j}\Bigr)^{-1} \Bigl(\sum_{j=1}^{n}s_{j}\mathbf{b}_{j}\Bigr) = \widetilde{L}(\mathbf{s}),
\]
which proves~\eqref{eq-alpay}.
\end{proof}

\begin{remark}\label{rem-real-realization}
Some multivariable realization theorems, including formulations such as~\cite{alpay2003}, are stated over the complex field. This causes no loss of generality here. If the rational function \(\widetilde L\) has real coefficients and is represented by complex matrices, then the standard realification of the state space, together with the corresponding realification of the input and output vectors, gives a real realization of the same scalar rational function, after a finite state-space augmentation; see \cite[Sec.~2]{futornyhornsergeichuk2008}. We shall therefore use real realizations without further comment.
\end{remark}

The next lemma records an elementary normalization of the closing column.

\begin{lemma}\label{lem-input-one}
Suppose
\[
  \widetilde{L}(\mathbf{s})=\boldeta\Bigl(\mathbf{I}_{N}-\sum_{j=1}^{n}s_j\mathbf{A}_j\Bigr)^{-1}\mathbf{b}
\]
is a non-zero realization of the form~\eqref{eq-alpay}. Then there is an equivalent realization of the same form, of the same dimension, for which the closing column is \(\mathbf{1}_N\).
\end{lemma}

\begin{proof}
Since \(\widetilde L\not\equiv0\), we may assume \(\mathbf b\ne0\). Choose an invertible matrix \(\mathbf H\in\R^{N\times N}\) such that \(\mathbf H\mathbf b=\mathbf 1_N\). Define
\[
  \widetilde{\boldeta}=\boldeta\mathbf H^{-1},\qquad
  \widetilde{\mathbf A}_j=\mathbf H\mathbf A_j\mathbf H^{-1},\qquad
  \widetilde{\mathbf b}=\mathbf H\mathbf b=\mathbf 1_N.
\]
Then
\[
  \widetilde{\boldeta}
  \Bigl(\mathbf I_N-\sum_{j=1}^n s_j\widetilde{\mathbf A}_j\Bigr)^{-1}
  \widetilde{\mathbf b}
  =
  \boldeta
  \Bigl(\mathbf I_N-\sum_{j=1}^n s_j\mathbf A_j\Bigr)^{-1}
  \mathbf b,
\]
so the represented rational function is unchanged.
\end{proof}

After normalizing the closing column, two lemmas convert the realization~\eqref{eq-alpay} into Kulkarni's resolvent. The first introduces a block structure that moves the dependence on $\mathbf{s}$ onto a diagonal matrix. The second augments the system to make the transition matrix invertible and Hurwitz-stable. The normalization of the input column to \(\mathbf 1_N\) also simplifies the terminal-column computation.

\begin{lemma}\label{lem-diag}
Let $\boldeta(\mathbf{I}_{N}-\sum_{j=1}^{n}s_{j}\mathbf{A}_{j})^{-1}\mathbf{1}_N$ be as in Theorem~\ref{thm-alpay}, after applying Lemma~\ref{lem-input-one}. Set $q=nN$. Define the $q\times q$ block-diagonal matrix
\[
  \mathbf{S}(\mathbf{s})=\diag(s_{1}\mathbf{I}_{N},\;s_{2}\mathbf{I}_{N},\;\ldots,\;s_{n}\mathbf{I}_{N}),
\]
the $q\times q$ block matrix $\mathbf{U}$ whose every block row of $N$ rows equals $[\mathbf{A}_{1},\;\mathbf{A}_{2},\;\ldots,\;\mathbf{A}_{n}]$,
\[
  \mathbf{U}=\begin{pmatrix} \mathbf{A}_{1}&\mathbf{A}_{2}&\cdots&\mathbf{A}_{n}\\ \mathbf{A}_{1}&\mathbf{A}_{2}&\cdots&\mathbf{A}_{n}\\ \vdots&&\ddots&\vdots\\ \mathbf{A}_{1}&\mathbf{A}_{2}&\cdots&\mathbf{A}_{n}
  \end{pmatrix},
\]
the block column $\mathbf{E}=(\mathbf{I}_{N},\mathbf{I}_{N},\ldots,\mathbf{I}_{N})^{\top}\in\R^{q\times N}$, the row vector $\balpha_{0}=(\boldeta,\;\mathbf{0}_{1\times(q-N)})\in\R^{1\times q}$, and the column $\mathbf{1}_{q}=\mathbf{E}\mathbf{1}_{N}$. Then for every \(\mathbf{s}\) such that
\[
  \det\Bigl(\mathbf{I}_{N}-\sum_{j=1}^{n}s_{j}\mathbf{A}_{j}\Bigr)\ne0,
\]
or, equivalently, such that \(\mathbf{I}_{q}-\mathbf{U}\mathbf{S}(\mathbf{s})\) is invertible,
\[
  \balpha_{0}(\mathbf{I}_{q}-\mathbf{U}\mathbf{S}(\mathbf{s}))^{-1}\mathbf{1}_{q} =\boldeta\Bigl(\mathbf{I}_{N}-\sum_{j=1}^{n}s_{j}\mathbf{A}_{j}\Bigr)^{-1}\mathbf{1}_{N}.
\]
\end{lemma}
\begin{proof}
The \(k\)-th block row of $\mathbf{U}\mathbf{S}(\mathbf{s})$ equals $(s_{1}\mathbf{A}_{1},\ldots,s_{n}\mathbf{A}_{n})$ for every \(k=1,\ldots,n\), so
\[
  \mathbf{U}\mathbf{S}(\mathbf{s})\mathbf{E} =
  \begin{pmatrix}
    \sum_{j=1}^{n}s_{j}\mathbf{A}_{j}\\ \sum_{j=1}^{n}s_{j}\mathbf{A}_{j}\\ \vdots\\ \sum_{j=1}^{n}s_{j}\mathbf{A}_{j}
  \end{pmatrix}
  =\mathbf{E}\sum_{j=1}^{n}s_{j}\mathbf{A}_{j}.
\]
Consequently,
\begin{equation}\label{eq-intertwining}
(\mathbf{I}_{q}-\mathbf{U}\mathbf{S}(\mathbf{s}))\mathbf{E} = \mathbf{E}\Bigl(\mathbf{I}_{N}-\sum_{j=1}^{n}s_{j}\mathbf{A}_{j}\Bigr).
\end{equation}

We next compare the two invertibility conditions. We use the determinant identity
\[
  \det(\mathbf{I}_{r}-\mathbf{C}\mathbf{D}) = \det(\mathbf{I}_{h}-\mathbf{D}\mathbf{C}), \qquad \mathbf{C}\in\R^{r\times h},\quad \mathbf{D}\in\R^{h\times r};
\]
this is the identity stated in~\cite[Problem~1.3.P28]{hornjohnson2013}. Applying it with
\[
  \mathbf{C}=\mathbf{E}, \qquad \mathbf{D}=(s_{1}\mathbf{A}_{1},\,s_{2}\mathbf{A}_{2},\,\ldots,\,s_{n}\mathbf{A}_{n})
\]
gives
\[
  \det\!\left( \mathbf{I}_{q} - \mathbf{E}(s_{1}\mathbf{A}_{1},\,s_{2}\mathbf{A}_{2},\,\ldots,\,s_{n}\mathbf{A}_{n}) \right) = \det\!\left( \mathbf{I}_{N} - (s_{1}\mathbf{A}_{1},\,s_{2}\mathbf{A}_{2},\,\ldots,\,s_{n}\mathbf{A}_{n})\mathbf{E} \right).
\]
Since
\[
  \mathbf{U}\mathbf{S}(\mathbf{s}) = \mathbf{E}(s_{1}\mathbf{A}_{1},\,s_{2}\mathbf{A}_{2},\,\ldots,\,s_{n}\mathbf{A}_{n})
\]
and
\[
  (s_{1}\mathbf{A}_{1},\,s_{2}\mathbf{A}_{2},\,\ldots,\,s_{n}\mathbf{A}_{n})\mathbf{E} = \sum_{j=1}^{n}s_{j}\mathbf{A}_{j},
\]
we obtain
\[
  \det(\mathbf{I}_{q}-\mathbf{U}\mathbf{S}(\mathbf{s})) = \det\Bigl(\mathbf{I}_{N}-\sum_{j=1}^{n}s_{j}\mathbf{A}_{j}\Bigr).
\]

This equality ensures that the two inverses exist for the same values of \(\mathbf{s}\). Moreover, at \(\mathbf{s}=\mathbf{0}\) the determinant on the right is \(\det(\mathbf{I}_{N})=1\). Hence, by continuity, both matrices are invertible for all \(\mathbf{s}\) in a sufficiently small neighbourhood of the origin. Multiplying \eqref{eq-intertwining} on the left by $(\mathbf{I}_{q}-\mathbf{U}\mathbf{S}(\mathbf{s}))^{-1}$ and on the right by $\bigl(\mathbf{I}_{N}-\sum_{j=1}^{n}s_j\mathbf{A}_j\bigr)^{-1}$ gives
\begin{equation}\label{eq-key}
  \mathbf{E}\Bigl(\mathbf{I}_{N}-\sum_{j=1}^{n}s_{j}\mathbf{A}_{j}\Bigr)^{-1} = (\mathbf{I}_{q}-\mathbf{U}\mathbf{S}(\mathbf{s}))^{-1}\mathbf{E}
\end{equation}
on this common domain of invertibility. Since $\mathbf{1}_{q}=\mathbf{E}\mathbf{1}_{N}$,
\[
  \mathbf{E}\Bigl(\mathbf{I}_{N}-\sum_{j=1}^{n}s_{j}\mathbf{A}_{j}\Bigr)^{-1}\mathbf{1}_{N} = (\mathbf{I}_{q}-\mathbf{U}\mathbf{S}(\mathbf{s}))^{-1}\mathbf{1}_{q}.
\]
Finally, $\balpha_{0}\mathbf{E}=\boldeta$ because $\balpha_{0}=(\boldeta,\mathbf{0}_{1\times(q-N)})$. Therefore,
\[
  \boldeta\Bigl(\mathbf{I}_{N}-\sum_{j=1}^{n}s_{j}\mathbf{A}_{j}\Bigr)^{-1}\mathbf{1}_{N} = \balpha_{0}(\mathbf{I}_{q}-\mathbf{U}\mathbf{S}(\mathbf{s}))^{-1}\mathbf{1}_{q},
\]
as claimed.
\end{proof}

The form $(\mathbf{I}_{q}-\mathbf{U}\mathbf{S}(\mathbf{s}))^{-1}$ is not yet in Kulkarni's format $(-\mathbf{T}+\Delta(\cdot))^{-1}$ with $\mathbf{T}$ invertible. The next lemma addresses this by embedding into a larger system.

\begin{lemma}\label{lem-invert}
Let $\balpha_{0},\mathbf{U},\mathbf{S}(\mathbf{s})$ be as in Lemma~\ref{lem-diag}. Set $\ell=2q$ and define
\[
  \mathbf{B}=\begin{pmatrix}\mathbf{U}&-\mathbf{I}_{q}\\(\mathbf{U}+\mathbf{I}_{q})^{2}&-\mathbf{U}-2\mathbf{I}_{q}\end{pmatrix} \in\R^{\ell\times \ell}.
\]
Set $\mathbf{T}=\mathbf{B}^{-1}$ and
\[
  \Delta_{\mathrm{aug}}(\mathbf{s}) =
  \begin{pmatrix}
    \mathbf{S}(\mathbf{s})&\mathbf{0}_{q\times q}\\ \mathbf{0}_{q\times q}&\mathbf{0}_{q\times q}
  \end{pmatrix}.
\]
Define
\[
  \widehat{\balpha}=(\balpha_{0},\;\mathbf{0}_{1\times q}), \qquad \mathbf{b}_{1}=\mathbf{1}_{\ell}=
  \begin{pmatrix}\mathbf{1}_{q}\\ \mathbf{1}_{q}\end{pmatrix}.
\]
For every \(\mathbf{s}\) such that \(\mathbf{I}_{q}-\mathbf{U}\mathbf{S}(\mathbf{s})\) is invertible, $\mathbf{T}$ is invertible, Hurwitz-stable, and
\[
  \widehat{\balpha}\bigl(-\mathbf{T}+\Delta_{\mathrm{aug}}(\mathbf{s})\bigr)^{-1}(-\mathbf{T}\mathbf{1}_{\ell}) =\balpha_{0}(\mathbf{I}_{q}-\mathbf{U}\mathbf{S}(\mathbf{s}))^{-1}\mathbf{1}_{q}.
\]
\end{lemma}

\begin{proof}
Direct block multiplication verifies that $\mathbf{B}$ is invertible with inverse
\[
  \mathbf{T}=\mathbf{B}^{-1}=\begin{pmatrix}-\mathbf{U}-2\mathbf{I}_{q}&\mathbf{I}_{q}\\-(\mathbf{U}+\mathbf{I}_{q})^{2}&\mathbf{U}\end{pmatrix}.
\]
In particular, $\mathbf{T}$ is invertible. We next verify that it is Hurwitz-stable. Since the blocks of $\mathbf{T}$ are polynomials in $\mathbf{U}$, they commute, and the standard determinant formula for a $2\times2$ block matrix with commuting blocks gives
\[
\begin{aligned}
  \det(\lambda\mathbf{I}_{\ell}-\mathbf{T})
  &=
  \det\!\left(
  ((\lambda+2)\mathbf{I}_{q}+\mathbf{U})(\lambda\mathbf{I}_{q}-\mathbf{U})
  +(\mathbf{U}+\mathbf{I}_{q})^{2}
  \right) \\
  &=
  \det\bigl((\lambda+1)^{2}\mathbf{I}_{q}\bigr)
  =
  (\lambda+1)^{2q}.
\end{aligned}
\]
Hence every eigenvalue of $\mathbf{T}$ is equal to $-1$, and therefore $\mathbf{T}$ is Hurwitz-stable. Since
\[
  -\mathbf{T}+\Delta_{\mathrm{aug}}(\mathbf{s}) = -\mathbf{B}^{-1}\bigl(\mathbf{I}_{\ell}-\mathbf{B}\Delta_{\mathrm{aug}}(\mathbf{s})\bigr),
\]
we have
\[
  \begin{aligned}
  \widehat{\balpha}\bigl(-\mathbf{T}+\Delta_{\mathrm{aug}}(\mathbf{s})\bigr)^{-1}(-\mathbf{T}\mathbf{1}_{\ell}) &= \widehat{\balpha} \bigl[-\mathbf{B}^{-1}\bigl(\mathbf{I}_{\ell}-\mathbf{B}\Delta_{\mathrm{aug}}(\mathbf{s})\bigr)\bigr]^{-1} (-\mathbf{B}^{-1}\mathbf{1}_{\ell})\\
  &= \widehat{\balpha} \bigl(\mathbf{I}_{\ell}-\mathbf{B}\Delta_{\mathrm{aug}}(\mathbf{s})\bigr)^{-1} (-\mathbf{B})(-\mathbf{B}^{-1}\mathbf{1}_{\ell})\\
  &= \widehat{\balpha} \bigl(\mathbf{I}_{\ell}-\mathbf{B}\Delta_{\mathrm{aug}}(\mathbf{s})\bigr)^{-1} \mathbf{1}_{\ell}.
  \end{aligned}
\]
Moreover,
\[
  \mathbf{B}\Delta_{\mathrm{aug}}(\mathbf{s}) =
  \begin{pmatrix}
    \mathbf{U}\mathbf{S}(\mathbf{s})&\mathbf{0}_{q\times q}\\ (\mathbf{U}+\mathbf{I}_{q})^{2}\mathbf{S}(\mathbf{s})&\mathbf{0}_{q\times q}
  \end{pmatrix},
\]
and hence
\[
  \mathbf{I}_{\ell}-\mathbf{B}\Delta_{\mathrm{aug}}(\mathbf{s}) =
  \begin{pmatrix}
    \mathbf{I}_{q}-\mathbf{U}\mathbf{S}(\mathbf{s})&\mathbf{0}_{q\times q}\\ -(\mathbf{U}+\mathbf{I}_{q})^{2}\mathbf{S}(\mathbf{s})&\mathbf{I}_{q}
  \end{pmatrix}.
\]
Using the block inverse formula for lower triangular block matrices,
\[
  \bigl(\mathbf{I}_{\ell}-\mathbf{B}\Delta_{\mathrm{aug}}(\mathbf{s})\bigr)^{-1} =
  \begin{pmatrix}
    (\mathbf{I}_{q}-\mathbf{U}\mathbf{S}(\mathbf{s}))^{-1}&\mathbf{0}_{q\times q}\\ (\mathbf{U}+\mathbf{I}_{q})^{2}\mathbf{S}(\mathbf{s})(\mathbf{I}_{q}-\mathbf{U}\mathbf{S}(\mathbf{s}))^{-1}&\mathbf{I}_{q}
  \end{pmatrix}.
\]
Because $\widehat{\balpha}=(\balpha_{0},\mathbf{0}_{1\times q})$ is supported on the first $q$ coordinates,
\[
  \widehat{\balpha} \bigl(\mathbf{I}_{\ell}-\mathbf{B}\Delta_{\mathrm{aug}}(\mathbf{s})\bigr)^{-1} \mathbf{1}_{\ell} = \balpha_{0}(\mathbf{I}_{q}-\mathbf{U}\mathbf{S}(\mathbf{s}))^{-1}\mathbf{1}_{q}.
\]
This concludes the proof.
\end{proof}

Combining the preceding lemmas, we now prove the main result of this section.
\begin{theorem}\label{thm-mvme}
For every $n\ge 2$, $\;\MVME{n}=\MME{n}$. More precisely, every $\bm{X}\in\MVME{n}$ admits an $\MME{n}$ representation $(\balpha,\mathbf{T},\mathbf{K})$ satisfying $\mathbf{t}=-\mathbf{T}\mathbf{1}$, $\;0\le\balpha\mathbf{1}\le 1$, and $\mathbf{K}\ge 0$.
\end{theorem}

\begin{proof}
The inclusion $\MME{n}\subseteq\MVME{n}$ is Remark~\ref{rem-trivial}. We prove $\MVME{n}\subseteq\MME{n}$. We begin by separating the possible atom at the origin from the strictly proper part of the transform.

Let $\bm{X}\in\MVME{n}$ with joint Laplace transform $L_{\bm{X}}(\mathbf{s})$. Let $p_{0}=\mathbb{P}(\bm{X}=\mathbf{0})$ be the atom at the origin, and write
\[
  L_{\bm{X}}(\mathbf{s})=p_{0}+\widetilde{L}(\mathbf{s}).
\]
Then \(\widetilde{L}\) is the part of the transform that vanishes when all transform variables tend to infinity. More precisely, for every $\mathbf r\in(0,\infty)^n$, dominated convergence gives $\widetilde L(\lambda\mathbf r)\to0$ as $\lambda\to\infty$; for a rational function this is the strict properness at infinity required in Theorem~\ref{thm-alpay}. If $\widetilde{L}\equiv 0$, then $L_{\bm{X}}\equiv p_{0}=1$, and the degenerate law at the origin belongs to $\MME{n}$ by choosing any Hurwitz-stable invertible $\mathbf{T}$ and setting $\balpha=0$. Hence assume $\widetilde{L}\not\equiv0$.
By Theorem~\ref{thm-alpay} and Lemma~\ref{lem-input-one}, we may fix a realization of the form
\[
  \widetilde{L}(\mathbf{s}) = \boldeta\Bigl(\mathbf{I}_{N}-\sum_{j=1}^{n}s_{j}\mathbf{A}_{j}\Bigr)^{-1}\mathbf{1}_{N}.
\]

Apply Lemma~\ref{lem-diag} to obtain
\[
  \widetilde{L}(\mathbf{s}) = \balpha_{0}(\mathbf{I}_{q}-\mathbf{U}\mathbf{S}(\mathbf{s}))^{-1}\mathbf{1}_{q}.
\]
Lemma~\ref{lem-invert} gives
\[
  \widetilde{L}(\mathbf{s})=\widehat{\balpha} (-\mathbf{T}+\Delta_{\mathrm{aug}}(\mathbf{s}))^{-1} (-\mathbf{T}\mathbf{1}_{\ell}),
\]
with $\mathbf{T}$ invertible and Hurwitz-stable. Set
\[
  \balpha=\widehat{\balpha},\qquad \mathbf t=-\mathbf T\mathbf 1_{\ell}.
\]

The matrix
\[
  \Delta_{\mathrm{aug}}(\mathbf{s}) = \diag(s_{1}\mathbf{I}_{N},\ldots,s_{n}\mathbf{I}_{N},\mathbf{0}_{q\times q})
\]
can be written in the reward form
\[
  \Delta_{\mathrm{aug}}(\mathbf{s})=\Delta(\mathbf{K}\mathbf{s}).
\]
Indeed, with \(q=nN\) and \(\ell=2q\), let
\[
  \mathbf{K} =
  \begin{pmatrix}
    \mathbf{I}_{n}\otimes \mathbf{1}_{N}\\ \mathbf{0}_{q\times n}
  \end{pmatrix}
  \in\R^{\ell\times n},
\]
where $\otimes$ denotes the Kronecker product between matrices or vectors. Then
\[
  \mathbf{K}\mathbf{s} =
  \begin{pmatrix}
    s_{1}\mathbf{1}_{N}\\ s_{2}\mathbf{1}_{N}\\ \vdots\\ s_{n}\mathbf{1}_{N}\\ \mathbf{0}_{q}
  \end{pmatrix},
\]
and therefore
\[
  \Delta(\mathbf{K}\mathbf{s}) = \diag(s_{1}\mathbf{I}_{N},\ldots,s_{n}\mathbf{I}_{N},\mathbf{0}_{q\times q}) = \Delta_{\mathrm{aug}}(\mathbf{s}).
\]
The equality of the original transform and the constructed resolvent is obtained first on a neighbourhood of the origin where all displayed inverses exist; since both sides are rational functions, it then holds identically on their common domain. Evaluating at $\mathbf{s}=0$ gives
\[
  1 = L_{\bm{X}}(0) = p_{0} + \balpha(-\mathbf{T})^{-1}\mathbf{t} = p_{0} + \balpha\mathbf{1}_{\ell},
\]
because $\mathbf{t}=-\mathbf{T}\mathbf{1}_{\ell}$. Hence $p_{0}=1-\balpha\mathbf{1}_{\ell}$, $0\le \balpha\mathbf{1}_{\ell}\le 1$. Therefore
\[
  L_{\bm{X}}(\mathbf{s}) = p_{0} + \balpha (-\mathbf{T}+\Delta(\mathbf{K}\mathbf{s}))^{-1} \mathbf{t}
\]
is a representation of the form~\eqref{eq-kulkarni} without sign constraints. The representation satisfies all three conditions of Remark~\ref{rem-normalization}: $\mathbf{t}=-\mathbf{T}\mathbf{1}_{\ell}$, $0\le\balpha\mathbf{1}_{\ell}\le 1$ by the displayed identity above, and $\mathbf{K}\ge 0$ since it is assembled from identity and zero blocks. This proves $\bm{X}\in\MME{n}$, and therefore $\MVME{n}=\MME{n}$.
\end{proof}

\section[Separation of MPH* and MVPH]{Separation of $\MPH{n}$ and $\MVPH{n}$}
\label{sec-separation}

The key ingredient is a factorization property for \(\MPH{n}\) that, to the best of our knowledge, does not appear to have been identified before. Whenever a transform is represented in Kulkarni's \(\MPH{n}\) form, the highest-degree part of its denominator must split into real linear factors with non-negative coefficients; the Markovian sign constraints on the generator and reward matrix impose this structure, as we will see below.

We start with the standard M-matrix fact needed in the argument. If \(\mathbf{T}\) is the subgenerator of a transient continuous-time Markov chain, then every principal submatrix of \(-\mathbf{T}\) is again a nonsingular M-matrix, and every principal minor of \(-\mathbf{T}\) is strictly positive~\cite[Chap.~6]{bermanplemmons1994}. The relevance for us is that, when the determinant $\det(-\mathbf{T}+\Delta(\mathbf{K}\mathbf{s}))$ is expanded, the highest-degree contribution is a principal minor of \(-\mathbf{T}\) multiplied by a product of reward linear forms. Since \(\mathbf{K}\) is non-negative and the relevant principal minor is strictly positive, this leading contribution cannot vanish through a zero coefficient.

For a non-zero polynomial \(F\) in the real variables \(s_{1},\ldots,s_{n}\), we write \(F_{\mathrm{top}}\) for its leading homogeneous part, obtained by retaining only those monomials whose total degree is maximal. For example, if $F(s_{1},s_{2},s_{3})=2+s_{1}-3s_{2}+s_{1}^{2}s_{3}+4s_{2}^{3}$, then the maximal total degree is \(3\), and hence $F_{\mathrm{top}}(s_{1},s_{2},s_{3})=s_{1}^{2}s_{3}+4s_{2}^{3}$. We shall use this operation to pass from a full denominator to its highest-degree contribution. The next elementary lemma records the only general fact about leading homogeneous parts that is needed later.

\begin{lemma}\label{lem-leading}
Let $F$ and $G$ be non-zero polynomials over the real field in the variables $s_{1},\ldots,s_{n}$. Suppose that $F=GH$ for another polynomial $H$. Denote the leading homogeneous parts by $F_{\mathrm{top}}$, $G_{\mathrm{top}}$, and $H_{\mathrm{top}}$. Then
\[
  F_{\mathrm{top}}=G_{\mathrm{top}}H_{\mathrm{top}}.
\]
In particular, the leading homogeneous part of $G$ is a factor of the leading homogeneous part of $F$.
\end{lemma}

\begin{proof}
Write $G=G_{\mathrm{top}}+G_{<}$ and $H=H_{\mathrm{top}}+H_{<}$, where the terms with subscript $<$ have strictly lower total degree. Then
\[
  GH = G_{\mathrm{top}}H_{\mathrm{top}} + G_{\mathrm{top}}H_{<} + G_{<}H_{\mathrm{top}} + G_{<}H_{<}.
\]
The term $G_{\mathrm{top}}H_{\mathrm{top}}$ is homogeneous of degree $\deg G+\deg H$, while every remaining term has strictly smaller total degree. Since the product of two non-zero polynomials over $\R$ is non-zero, the leading homogeneous part of $F=GH$ is $G_{\mathrm{top}}H_{\mathrm{top}}$.
\end{proof}

\begin{theorem}[$\MPH{n}$ factorization]\label{thm-invariant}
Let $\bm{Y}\in\MPH{n}$ with joint Laplace transform $L_{\bm{Y}}(\mathbf{s})=P(\mathbf{s})/Q(\mathbf{s})$ in minimal form. After multiplying numerator and denominator by a non-zero scalar if necessary, the leading homogeneous part $Q_{\mathrm{top}}(\mathbf{s})$ of $Q$ factors as a product of real linear functions with non-negative coefficients,
\[
  Q_{\mathrm{top}}(\mathbf{s})=c_{0}\prod_{i=1}^{\nu}\ell_{i}(\mathbf{s}),\qquad \ell_{i}(\mathbf{s})=\sum_{j=1}^{n}\lambda_{ij}s_{j},\quad\lambda_{ij}\ge 0,
\]
for some $c_{0}>0$ (with $\nu=0$ and the empty product allowed when the minimal denominator is constant).
\end{theorem}

\begin{proof}
Let an \(m\)-state reward-based representation of \(\bm{Y}\) be given. Thus \(\balpha\in\R_{+}^{1\times m}\) is a sub-probability row vector, \(\mathbf{T}\in\R^{m\times m}\) is a transient subgenerator, \(\mathbf{K}\in\R_{+}^{m\times n}\) is the reward matrix, \(\mathbf{t}=-\mathbf{T}\mathbf{1}\), and
\[
  L_{\bm{Y}}(\mathbf{s}) = p_{0} + \balpha \bigl(-\mathbf{T}+\Delta(\mathbf{K}\mathbf{s})\bigr)^{-1} \mathbf{t}.
\]
Set
\[
  F(\mathbf{s}) = \det\bigl(-\mathbf{T}+\Delta(\mathbf{K}\mathbf{s})\bigr).
\]
Using the adjugate identity \(\mathbf{M}^{-1}=(\det\mathbf{M})^{-1}\operatorname{adj}(\mathbf{M})\) for \(\mathbf{M}=-\mathbf{T}+\Delta(\mathbf{K}\mathbf{s})\), the transform can be written as
\[
  L_{\bm{Y}}(\mathbf{s}) = \frac{ p_{0}F(\mathbf{s}) + \balpha\,\operatorname{adj}\bigl(-\mathbf{T}+\Delta(\mathbf{K}\mathbf{s})\bigr)\mathbf{t} }{ F(\mathbf{s}) }.
\]
Because the rational function \(P(\mathbf{s})/Q(\mathbf{s})\) is in minimal form, \(P\) and \(Q\) are coprime. Therefore the minimal denominator \(Q(\mathbf{s})\) is a factor of \(F(\mathbf{s})\).

Let $I\subseteq\{1,\ldots,m\}$ denote the set of indices $i$ for which the reward row $\mathbf{K}_{i,\cdot}$ is not identically zero. For each $i\in I$, write
\[
  \kappa_{i}(\mathbf{s})=\sum_{j=1}^{n}\mathbf{K}_{ij}s_{j}.
\]
The \(\MPH{n}\) constraint gives \(\mathbf{K}_{ij}\ge0\), so every coefficient of \(\kappa_i\) is non-negative. Since \(i\in I\), at least one of these coefficients is positive.

We use the identity
\[
  \det(A+\diag(d_1,\ldots,d_m))
  =
  \sum_{J\subseteq\{1,\ldots,m\}}
  \det(A_{J^c,J^c})\prod_{i\in J}d_i,
\]
where \(A_{J^c,J^c}\) denotes the principal submatrix obtained by retaining the rows and columns indexed by \(J^c\), and where the determinant of the empty matrix and the empty product are both taken to be \(1\). This identity follows directly from the determinant expansion by collecting terms according to the set of diagonal entries \(d_i\) selected. Applying it with \(A=-\mathbf T\) and \(d_i=\kappa_i(\mathbf s)\), where \(\kappa_i(\mathbf s)=0\) for \(i\notin I\), gives
\[
  F(\mathbf{s}) = \sum_{J\subseteq\{1,\ldots,m\}} \det\bigl((-\mathbf{T})_{J^{c},J^{c}}\bigr) \prod_{i\in J}\kappa_i(\mathbf{s}).
\]

Since $\kappa_i\equiv0$ for $i\notin I$, the non-zero terms in this expansion have total degree at most $|I|$. The unique term of degree $|I|$ is obtained by taking $J=I$. Therefore
\[
  F_{\mathrm{top}}(\mathbf{s}) = \det\bigl((-\mathbf{T})_{I^{c},I^{c}}\bigr) \prod_{i\in I}\kappa_i(\mathbf{s}).
\]
By the M-matrix fact recalled at the beginning of the section, the coefficient $\det\bigl((-\mathbf{T})_{I^{c},I^{c}}\bigr)$ is strictly positive, again with the convention that the determinant of the empty matrix is \(1\). Hence \(F_{\mathrm{top}}\) is a positive scalar multiple of a product of real linear functions with non-negative coefficients.

Since \(Q\) divides \(F\), Lemma~\ref{lem-leading} implies that \(Q_{\mathrm{top}}\) divides \(F_{\mathrm{top}}\). Consequently \(Q_{\mathrm{top}}\) is a non-zero scalar multiple of a product of some of the linear factors \(\kappa_i\). Multiplying \(P\) and \(Q\) by a non-zero scalar if necessary, this scalar may be chosen positive. This gives
\[
  Q_{\mathrm{top}}(\mathbf{s}) = c_{0}\prod_{i=1}^{\nu}\ell_i(\mathbf{s}), \qquad c_{0}>0,
\]
where each \(\ell_i\) is one of the linear functions \(\kappa_r\), possibly after relabelling. Each \(\ell_i\) therefore has non-negative coefficients. This proves the stated factorization.
\end{proof}
We now construct an explicit distribution in \(\MVPH{3}\setminus\MPH{3}\) from a central Wishart matrix. To the best of our knowledge, this connection between Wishart trace functionals and the projection-based multivariate phase-type class has not been explored before. We use only standard facts about the Wishart distribution, namely its density and determinant-form transform~\cite[Eq.~(3.2.1) and Eq.~(3.3.9)]{guptanagar2000}.

In the notation \(\mathcal{W}_{2}(2,\tfrac12 \mathbf{I}_2)\), the first subscript indicates that the random matrix is \(2\times2\), the first parameter is the number of degrees of freedom, and the second parameter is the scale matrix. Thus a random matrix \(\bm{Z}\sim\mathcal{W}_{2}(2,\tfrac12\mathbf{I}_{2})\) is supported on the cone of \(2\times2\) symmetric positive semidefinite matrices. In the present case the distribution is absolutely continuous on the positive-definite cone. The boundary consists of the singular positive semidefinite matrices, equivalently those with determinant zero, and has probability zero. Its density, with respect to Lebesgue measure on the three free coordinates \((z_{11},z_{22},z_{12})\), with \(z_{21}=z_{12}\) determined by symmetry, is
\[
  f_{\bm{Z}}(\mathbf{z}) = \pi^{-1}\det(\mathbf{z})^{-1/2}e^{-\Tr(\mathbf{z})}, \qquad \mathbf{z}=
  \begin{pmatrix}
    z_{11}&z_{12}\\ z_{12}&z_{22}
  \end{pmatrix}>0.
\]
Equivalently, one may realize
\[
  \bm{Z}=\frac12(\bm{G}_{1}\bm{G}_{1}^{\top}+\bm{G}_{2}\bm{G}_{2}^{\top}),
\]
where \(\bm{G}_{1},\bm{G}_{2}\) are independent standard Gaussian column vectors in \(\R^2\). From the Wishart transform, for every symmetric positive semidefinite matrix \(\mathbf{B}\),
\[
  \E\bigl[e^{-\Tr(\mathbf{B}\bm{Z})}\bigr]=\det(\mathbf{I}_2+\mathbf{B})^{-1}.
\]

Let $\bm{Z}\sim\mathcal{W}_{2}(2,\tfrac{1}{2}\mathbf{I}_{2})$. Fix the symmetric positive-definite matrices
\[
  \mathbf{H}_{1}=\begin{pmatrix}1&0\\0&1\end{pmatrix},\quad \mathbf{H}_{2}=\begin{pmatrix}2&1\\1&2\end{pmatrix},\quad \mathbf{H}_{3}=\begin{pmatrix}3&0\\0&1\end{pmatrix},
\]
and define $\bm{X}=(X_{1},X_{2},X_{3})^{\top}$ by $X_{j}=\Tr(\mathbf{H}_{j}\bm{Z})$. This choice of \(\mathbf{H}_{1},\mathbf{H}_{2},\mathbf{H}_{3}\) has two useful features.
First, every non-zero non-negative linear combination of \(\mathbf H_1,\mathbf H_2,\mathbf H_3\) is positive definite. For every \(\mathbf a\in\R_+^3\setminus\{0\}\), define
\begin{equation}\label{eq:Ha123}
  \mathbf H(\mathbf a)=a_1\mathbf H_1+a_2\mathbf H_2+a_3\mathbf H_3.
\end{equation}
Then \(\mathbf H(\mathbf a)\) is positive definite, and, by linearity of the trace,
\[
  \langle \mathbf a,\bm X\rangle
  =
  \sum_{j=1}^3 a_j\Tr(\mathbf H_j\bm Z)
  =
  \Tr(\mathbf H(\mathbf a)\bm Z).
\]
The Wishart transform therefore gives, for \(u\ge0\),
\[
  \E\!\left[e^{-u\langle \mathbf a,\bm X\rangle}\right]
  =
  \E\!\left[e^{-\Tr(u\mathbf H(\mathbf a)\bm Z)}\right]
  =
  \det(\mathbf I_2+u\mathbf H(\mathbf a))^{-1}.
\]
This explicit one-dimensional projection transform will be used below to show that \(\bm X\in\MVPH{3}\). Second, the determinant polynomial
\[
  \det(\mathbf I_2+s_1\mathbf H_1+s_2\mathbf H_2+s_3\mathbf H_3)
\]
has a leading homogeneous part that does not factor into real linear forms. This is exactly the feature needed to violate the factorization criterion of Theorem~\ref{thm-invariant}. The verification of this second point is given in Theorem~\ref{thm-notmph}.

\begin{lemma}[Density of $\bm{X}$]\label{lem-density}
The random vector $\bm{X}$ has a density on $\R_{+}^{3}$ given by
\[
  f_{\bm{X}}(\mathbf{x})=\frac{1}{2\pi}\, \bigl(-7x_{1}^{2}-x_{2}^{2}-x_{3}^{2}+4x_{1}x_{2}+4x_{1}x_{3}\bigr)^{-1/2} e^{-x_{1}}
\]
on the set
\[
  \mathcal{S} = \left\{ \mathbf{x}\in\R_{+}^{3}: -7x_{1}^{2}-x_{2}^{2}-x_{3}^{2}+4x_{1}x_{2}+4x_{1}x_{3}>0 \right\},
\]
and \(f_{\bm{X}}(\mathbf{x})=0\) on \(\R_{+}^{3}\setminus\mathcal{S}\).
\end{lemma}

\begin{proof}
Write a deterministic symmetric matrix in the support of \(\bm{Z}\) as
\[
  \mathbf{z} =
  \begin{pmatrix}
    z_{11}&z_{12}\\ z_{12}&z_{22}
  \end{pmatrix}.
\]
The linear map from \(\mathbf{z}\) to \(\mathbf{x}=(x_1,x_2,x_3)^{\top}\) is
\[
  x_{1}=z_{11}+z_{22},\qquad x_{2}=2z_{11}+2z_{22}+2z_{12},\qquad x_{3}=3z_{11}+z_{22}.
\]
It is invertible, with inverse
\[
  z_{11}=\frac{x_{3}-x_{1}}{2},\qquad z_{22}=\frac{3x_{1}-x_{3}}{2},\qquad z_{12}=\frac{x_{2}-2x_{1}}{2}.
\]
The absolute value of the Jacobian determinant of the inverse map \(\mathbf{x}\mapsto(z_{11},z_{22},z_{12})\) is \(1/4\). Moreover,
\[
  \Tr(\mathbf{z})=z_{11}+z_{22}=x_1
\]
and
\[
  4\det(\mathbf{z}) = 4(z_{11}z_{22}-z_{12}^{2}) = -7x_{1}^{2}-x_{2}^{2}-x_{3}^{2}+4x_{1}x_{2}+4x_{1}x_{3}.
\]
Substitution in the change-of-variables formula gives
\[
\begin{aligned}
  f_{\bm{X}}(\mathbf{x}) &= f_{\bm{Z}}(\mathbf{z}(\mathbf{x}))\, \frac14  \\
  &= \pi^{-1} \left( \frac{-7x_{1}^{2}-x_{2}^{2}-x_{3}^{2}+4x_{1}x_{2}+4x_{1}x_{3}}{4} \right)^{-1/2} e^{-x_1}\,\frac14 \\
  &= \frac{1}{2\pi} \bigl(-7x_{1}^{2}-x_{2}^{2}-x_{3}^{2}+4x_{1}x_{2}+4x_{1}x_{3}\bigr)^{-1/2} e^{-x_1}.
\end{aligned}
\]
It remains only to describe the support in \(\mathbf{x}\)-coordinates. For a real symmetric \(2\times2\) matrix, positive semidefiniteness is equivalent to non-negative trace and non-negative determinant. Hence the closure of the support is determined by
\[
  x_1\ge0, \qquad -7x_{1}^{2}-x_{2}^{2}-x_{3}^{2}+4x_{1}x_{2}+4x_{1}x_{3}\ge0.
\]
The second inequality is equivalently
\[
  (x_{2}-2x_{1})^{2}+(x_{3}-2x_{1})^{2}\le x_{1}^{2}.
\]
This immediately implies
\[
  |x_{2}-2x_{1}|\le x_{1}, \qquad |x_{3}-2x_{1}|\le x_{1}.
\]
Since \(x_{1}\ge0\), these inequalities give
\[
  x_{1}\le x_{2}\le 3x_{1}, \qquad x_{1}\le x_{3}\le 3x_{1}.
\]
In particular \(x_{2}\ge0\) and \(x_{3}\ge0\), so the support described by the trace and determinant conditions is contained in \(\R_{+}^{3}\). The density formula itself is taken on the interior, where the determinant is strictly positive; the boundary has Lebesgue measure zero and probability zero.
\end{proof}

\begin{remark}\label{rem-support-obstruction}
The support description in Lemma~\ref{lem-density} already suggests another route to the conclusion that \(\bm X\notin\MPH{3}\). In a finite-state reward-based representation, conditional on a fixed path of the underlying continuous-time Markov chain, the reward vector is a non-negative linear combination of the finitely many reward vectors visited along that path. Taking the union over all possible visited state sets gives a finite union of polyhedral cones. Thus the closed support of an \(\MPH{3}\) distribution has a piecewise polyhedral geometry. By contrast, the closure of the support in Lemma~\ref{lem-density} is the quadratic cone
\[
  (x_2-2x_1)^2+(x_3-2x_1)^2\le x_1^2,\qquad x_1\ge0,
\]
whose section at \(x_1=1\) is a disk, not a finite union of polytopes. This geometric observation can itself be used as a route to proving non-membership in \(\MPH{3}\). In what follows, however, we use the algebraic denominator obstruction from Theorem~\ref{thm-invariant}, which can be checked directly from the Laplace transform and does not require identifying the density in closed form.
\end{remark}

\begin{theorem}[$\bm{X}\notin\MPH{3}$]\label{thm-notmph}
The distribution of $\bm{X}$ admits no finite-dimensional $\MPH{3}$ representation.
\end{theorem}

\begin{proof}
From the Wishart transform,
\[
  L_{\bm{X}}(\mathbf{s}) = \det\bigl(\mathbf{I}_{2}+s_{1}\mathbf{H}_{1}+s_{2}\mathbf{H}_{2}+s_{3}\mathbf{H}_{3}\bigr)^{-1}.
\]
Since the numerator is the constant polynomial \(1\), the minimal denominator is
\[
\begin{aligned}
  Q(\mathbf{s})
  &=\det\!\begin{pmatrix}
    1+s_1+2s_2+3s_3 & s_2\\
    s_2 & 1+s_1+2s_2+s_3
  \end{pmatrix}  \\
  &=(1+s_{1}+2s_{2}+3s_{3})(1+s_{1}+2s_{2}+s_{3})-s_{2}^{2}.
\end{aligned}
\]
Expanding,
\[
\begin{aligned}
  Q(\mathbf{s}) &= 1+2s_{1}+4s_{2}+4s_{3}  \\
  &\quad +s_{1}^{2}+4s_{1}s_{2}+4s_{1}s_{3}
  +3s_{2}^{2}+8s_{2}s_{3}+3s_{3}^{2}.
\end{aligned}
\]
Hence the leading homogeneous part is
\[
  Q_{2}(\mathbf{s}) =
  s_{1}^{2}+4s_{1}s_{2}+4s_{1}s_{3}
  +3s_{2}^{2}+8s_{2}s_{3}+3s_{3}^{2}.
\]

We next show that \(Q_{2}\) cannot be written as a product of two real linear functions. Suppose, to the contrary, that such a factorization exists. Since the coefficient of \(s_1^2\) in \(Q_2\) is \(1\), any factorization of \(Q_{2}\) may be written in the form
\[
  Q_{2}(\mathbf{s}) = (s_{1}+u s_{2}+v s_{3})(s_{1}+r s_{2}+t s_{3})
\]
for some real numbers \(u,v,r,t\). Comparing coefficients gives
\[
  u+r=4,\qquad ur=3,\qquad v+t=4,\qquad vt=3,\qquad ut+vr=8.
\]
The first two equations imply that \(u\) and \(r\) are \(1\) and \(3\), in some order. Similarly, the next two equations imply that \(v\) and \(t\) are \(1\) and \(3\), in some order. Therefore the mixed coefficient \(ut+vr\) can only be
\[
1\cdot 3+1\cdot 3=6
\qquad \text{or} \qquad
1\cdot 1+3\cdot 3=10,
\]
depending on whether the two orderings are opposite or the same. Neither value is \(8\), a contradiction. Thus \(Q_2\) has no factorization into real linear functions.

By Theorem~\ref{thm-invariant}, the leading homogeneous part of the minimal denominator of any \(\MPH{3}\) distribution must factor into real linear functions with non-negative coefficients. Since \(Q_2\) does not even factor into real linear functions, \(\bm X\notin\MPH{3}\).
\end{proof}

\begin{theorem}[$\bm{X}\in\MVPH{3}$]\label{thm-mvph}
The distribution of $\bm{X}$ belongs to $\MVPH{3}$.
\end{theorem}

\begin{proof}
The Wishart transform displayed above gives, for all $\mathbf{s}$ in a neighbourhood of the non-negative orthant,
\begin{equation}\label{eq-ltx}
  L_{\bm{X}}(\mathbf{s}) =
  \E\bigl[e^{-\Tr((s_{1}\mathbf{H}_{1}+s_{2}\mathbf{H}_{2}+s_{3}\mathbf{H}_{3})\bm{Z})}\bigr]
  =
  \det\bigl(\mathbf{I}_{2}+s_{1}\mathbf{H}_{1}+s_{2}\mathbf{H}_{2}+s_{3}\mathbf{H}_{3}\bigr)^{-1}.
\end{equation}

To verify membership in \(\MVPH{3}\), we must check every non-negative one-dimensional projection. Fix
\[
  \mathbf{a}=(a_{1},a_{2},a_{3})^{\top}\in\R_{+}^{3}\setminus\{0\}.
\]
Using the notation from~\eqref{eq:Ha123}, linearity of the trace gives
\[
  \langle \mathbf{a},\bm{X}\rangle
  =
  \sum_{j=1}^{3}a_j\Tr(\mathbf{H}_{j}\bm{Z})
  =
  \Tr\!\left(\mathbf{H}(\mathbf{a})\bm{Z}\right).
\]
Each \(\mathbf{H}_{j}\) is positive definite, and the coefficients \(a_j\) are non-negative with at least one \(a_j>0\). Hence \(\mathbf{H}(\mathbf{a})\) is positive definite. Therefore its two eigenvalues, denoted by $\lambda_{1}(\mathbf{a})$, $\lambda_{2}(\mathbf{a})$, are strictly positive.

Evaluating~\eqref{eq-ltx} along the ray $\mathbf{s}=u\mathbf{a}$ gives
\[
  L_{\langle \mathbf{a},\bm{X}\rangle}(u)
  =
  L_{\bm{X}}(u\mathbf{a})
  =
  \det\bigl(\mathbf{I}_{2}+u\mathbf{H}(\mathbf{a})\bigr)^{-1}
  =
  \bigl(1+u\lambda_{1}(\mathbf{a})\bigr)^{-1}
  \bigl(1+u\lambda_{2}(\mathbf{a})\bigr)^{-1}.
\]

This is the Laplace transform of the sum of two independent exponential random variables with scale parameters \(\lambda_{1}(\mathbf{a})\) and \(\lambda_{2}(\mathbf{a})\), equivalently rates \(1/\lambda_{1}(\mathbf a)\) and \(1/\lambda_{2}(\mathbf a)\). Such a distribution is phase-type. Since $\langle \mathbf{a},\bm{X}\rangle\in\PH$ for every $\mathbf{a}\in\R_{+}^{3}\setminus\{0\}$, the definition of the projection-based class gives $\bm{X}\in\MVPH{3}$.
\end{proof}

\begin{corollary}\label{cor-alln}
For every $n\ge 3$, $\;\MPH{n}\subsetneq\MVPH{n}$.
\end{corollary}

\begin{proof}
For $n=3$ the strict inclusion is immediate from Theorems~\ref{thm-notmph} and~\ref{thm-mvph}. For $n>3$, let $\bm{X}\in\MVPH{3}\setminus\MPH{3}$ be as above and define $\bm{Y}=(X_{1},X_{2},X_{3},E_{4},\ldots,E_{n})^{\top}$, where $E_{4},\ldots,E_{n}\overset{\mathrm{i.i.d.}}{\sim}\mathrm{Exp}(1)$ are independent of $\bm{X}$. For any $\mathbf{a}\in\R_{+}^{n}\setminus\{0\}$, the projection $\langle \mathbf{a},\bm{Y}\rangle$ is the sum of $\langle(a_{1},a_{2},a_{3})^{\top},\bm{X}\rangle\in\PH$ and a non-negative linear combination of independent exponentials, hence is phase-type. Here coefficients equal to zero simply give degenerate zero summands, and positive coefficients rescale exponentials; the PH class is closed under convolution of independent PH variables. Thus $\bm{Y}\in\MVPH{n}$.

If $\bm{Y}$ had an $\MPH{n}$ representation with reward matrix $\mathbf{K}$, then setting $s_4=\cdots=s_n=0$ in the joint transform would keep only the first three reward columns. This would give an $\MPH{3}$ representation for $(Y_1,Y_2,Y_3)=\bm{X}$, contradicting Theorem~\ref{thm-notmph}. Hence $\bm{Y}\notin\MPH{n}$.
\end{proof}

The preceding argument does not resolve the bivariate case. The reason is that the obstruction used above is intrinsically three-dimensional. In the trivariate example, non-membership in \(\MPH{3}\) is detected by an irreducible quadratic leading form, whereas the corresponding leading-form obstruction has no comparable strength in two variables. Thus the invariant developed here cannot distinguish \(\MPH{2}\) from \(\MVPH{2}\). Whether the two classes coincide in the bivariate case remains open.

\section{Concluding remarks}
\label{sec-remarks}

This paper has settled two questions on the structure of multivariate matrix-exponential and phase-type classes. For the matrix-exponential pair, the equality $\MVME{n}=\MME{n}$ shows that Kulkarni's algebraic resolvent, once freed from positivity, covers the full class of proper rational multivariate Laplace transforms. The resulting \(\MME{n}\) representations may moreover be chosen with non-negative reward matrix, normalized terminal column \(\mathbf t=-\mathbf T\mathbf 1\), and \(0\le\balpha\mathbf 1\le1\), so the transition matrix is the only ingredient that falls outside the Markovian constraints. The representations therefore retain much of the formal structure of \(\MPH{n}\) and may carry probabilistic interpretations not yet identified in the purely algebraic setting. For the phase-type pair, the M-matrix structure of any transient subgenerator is restrictive. It forces a factorization condition on the leading homogeneous part of the denominator that the Wishart trace construction violates for $n\ge 3$, establishing strict separation throughout that range.

On the PH side, the Markovian constraints impose a condition visible in the denominator algebra of the joint transform. This gives a concrete way to test candidate multivariate PH models and suggests that further algebraic criteria may be found by studying the denominators arising from reward-based representations.

The Wishart example shows that trace functionals of random matrices can produce natural elements of the projection-based class \(\MVPH{}\) without arising from a finite-state reward accumulation mechanism. This suggests that matrix-variate distributions may provide a useful source of examples and counterexamples for multivariate PH theory, particularly when their determinant-form transforms interact with positivity of one-dimensional projections. The Wishart trace density contains a genuine square-root factor despite having a rational Laplace transform. This behavior is absent from the standard univariate PH/ME density representation, where matrix exponentials give exponential-polynomial terms rather than algebraic square-root singularities. What density shapes, singularities, and support geometries can arise in \(\MVPH{}\) and \(\MVME{}\) more broadly remains unclear; the projection-based and reward-based classes may differ further along these lines.

\bibliographystyle{plain}
\bibliography{MVME_MVPH_references}

\end{document}